\documentclass[12pt, reqno]{amsart}
\usepackage{amssymb,latexsym,amsmath,amsfonts, amsthm}
\usepackage{latexsym}
\usepackage[numbers,sort&compress]{natbib}
\usepackage{color}
\usepackage[mathscr]{eucal}
\usepackage{comment}
\usepackage[linkcolor=blue, citecolor=black]{hyperref}
\hypersetup{colorlinks=true, urlcolor=blue}
\usepackage[nameinlink]{cleveref}
\usepackage{enumerate}
\usepackage[dvipsnames]{xcolor}

\numberwithin{equation}{section}

\theoremstyle{definition}
\newtheorem{definition}{Definition}[section]
\newtheorem{example}[definition]{Example}

\newtheorem{remark}[definition]{Remark}
\theoremstyle{plain}
\newtheorem{theorem}[definition]{Theorem}
\newtheorem{lemma}[definition]{Lemma}

\newtheorem{result}[definition]{Result}

\newtheorem{corollary}[definition]{Corollary}

\newtheorem{problem}[definition]{Problem}

\newcommand{\eps}{\varepsilon}

\newcommand{\rl}{{\sf Re}}

\begin{document}
	\title[On Spirallike Circularlike domain]{On Hyperbolicity of Spirallike Circularlike domain}
	\author{Sanjoy Chatterjee and Golam Mostafa Mondal}
	\address{Department of Mathematics and Statistics, Indian Institute of Science Education and Research Kolkata,
		Mohanpur -- 741 246}
	\email{ramvivsar@gmail.com}
	
	\address{Department of Mathematics, Indian Institute of Science, Bangalore-- 560012, India}
	
	\email{golammostafaa@gmail.com, golammondal@iisc.ac.in}
	
	\thanks{}
	\keywords{Kobayashi hyperbolic; Brody hyperbolic; Core of a domain; Spirallike domain; Circularlike domain}
	\subjclass[2020]{Primary: 32F45; 32H02; 32Q02}

	\date{\today}

	\begin{abstract}
	In this paper, we prove that a spirallike circularlike domain is Kobayashi hyperbolic if and only if its core is empty. In particular, we show that such a domain is Kobayashi hyperbolic if and only if it is (biholomorphic to) a bounded domain. We also propose a problem in this area.
	
 %%In this paper, we provide a class Kobayashi hyperbolic domain which has a bounded realization.
 
 \end{abstract}

	\maketitle
	
	\section{Introduction}
The aim of this paper is to establish a necessary and sufficient condition for certain classes of domains to be Kobayashi hyperbolic, based on the core of the domain. Given a domain $D \subset \mathbb{C}^d$, let $C_{D} : D \times D \rightarrow [0, \infty)$ denote the Carath$\acute{e}$odory pseudo-distance, and let $K_{D} : D \times D \rightarrow [0, \infty)$ denote the Kobayashi pseudo-distance (see \cite{JPbook2013} for the definition). We say that $D$ is Kobayashi (resp. Carath$\acute{e}$odory) hyperbolic if $K_{D}$ (resp. $C_{D}$) is an actual distance. Note that not every Kobayashi hyperbolic domain admits a bounded representation, i.e., not all Kobayashi hyperbolic domains are biholomorphic to bounded domains. One of the obstacles for the existence of such representations is the existence of core of the domain. The concept of the core, denoted as $\mathfrak{c}(D)$, for a domain $D \subset \mathbb{C}^d$ (or more generally, for a domain in a complex manifold $M$), was introduced and extensively studied in Harz, Shcherbina and Tomassini \cite{HaShTo2017,HaShTo2020}, and Poletky and Shcherbina \cite{PoSh2019}. It can be defined as follows:
	
	\begin{definition}\label{Def:Core_Smooth}
		Let $M$ be a complex manifold, and $D \subset M$ be a domain. Then, the set 
		\begin{align*}
			\mathfrak{c}(D) := \{\zeta \in D : \textit{every continuous plurisubharmonic function on $D$ that is bounded}\\{ \textit{from above fails to be strictly plurisubharmonic at }} \zeta\}
	\end{align*}
		is called the core of $D$.
	\end{definition}
	\noindent 
Consider a domain $D$ in $\mathbb{C}^d$. If $\mathfrak{c}(D)$ is empty, then every holomorphic mapping from $\mathbb{C}$ into $D$ is constant (such a domain is called \textit{Brody hyperbolic}). However, the converse is not always true in general, because there exists an unbounded Kobayashi hyperbolic domain with non-empty core (see \cite{GauShcher2021}). Again, if $D$ is bounded, then since $z\to \|z\|^2$ is a bounded strictly plurisubharmonic function on $D$, $\mathfrak{c}(D)$ is empty. By biholomorphic invariance of $\mathfrak{c}(D)$, we can say that $\mathfrak{c}(D)$ is empty if it is biholomorphic to a bounded domain. However, %there exists an unbounded domain with an empty core (unbounded realization of unit ball in $\mathbb{C}^n$), and 
there exists an unbounded Kobayashi hyperbolic domain with an empty core which does not have a bounded realization (see \Cref{Exam:Core_empty}). %Again, if $D$ is Kobayashi hyperbolic, then there does not exist any non-constant holomorphic function from $\mathbb{C}$ to $D$. On the other hand, there exists a non-hyperbolic domain $D$ such that every holomorphic function from $\mathbb{C}$ to $D$ is constant. An example of this is provided by Barth \cite{Barth1980}. 
In this paper, we study certain classes of domains in $\mathbb{C}^d$ for which the converse holds in each of the above cases.

%\begin{question}
%Does there exist a spirallike circularlike Brody hyperbolic domain that has a non-empty core or is non-hyperbolic?
%\end{question}

	\smallskip	
	
In \cite{Kod1982}, Kodama proved that a starlike circular domain in $\mathbb{C}^d$ is hyperbolic if and only if it is bounded. This present paper extends this result to a wide class of domains in $\mathbb{C}^d$. In particular, we show that the same conclusion can be made for spirallike circularlike domains in $\mathbb{C}^{d}$.
We need the following definitions to state our result.

\smallskip
\par A holomorphic vector field $V$ on a domain  $D\subseteq \mathbb{C}^{d}$ is a real vector field on $D$  such that  $$V(z)=\sum_{i=1}^{d} a_{i}(z) \frac{\partial }{\partial  x_{i}}+b_{i}(z) \frac{\partial }{\partial  y_{i}},$$ where $(a_j(z)+ib_{j}(z))$ is holomorphic function on $D$  for all $j \in \{1,2, \cdots ,d\}$, and  $z_{j}=x_{j}+iy_{j}$ are coordinates of $\mathbb{C}^{d}.$ Any holomorphic map $f:\mathbb{C}^d \to \mathbb{C}^d$ can be thought as a holomorphic vector field. The set of all holomorphic vector fields on $D$ is denoted by $ \mathfrak{X}_{\mathcal{O}}(D)$. A point $z_{0}\in D$ is said to be an equilibrium point of the vector field $V$ if $V(z_{0})=0$. A holomorphic system of differential equations associated to $V \in \mathfrak{X}_{\mathcal{O}}(\mathbb{C}^d) $ is as follows:
\begin{align}\label{E:diffeq1}
		\dot{x}(z)&=\widetilde{V}(x(z))\notag\\
		x(0)&=z_{0},
	\end{align}
	 where $z=t+is$,     and $\widetilde{V}(z)=\sum_{j=1}^{d}(a_{j}(z)+ib_{j}(z))\frac{\partial}{\partial z_{j}}$.

A vector field $V \in \mathfrak{X}_{\mathcal{O}}(\mathbb{C}^d) $ is said to be complete if the solution of the system \eqref{E:diffeq1} exists for all $z_{0} \in \mathbb{C}^d$ and  $t \in \mathbb{R}$. It follows from \cite[Theorem]{AMR2000} that the solution of the system \eqref{E:diffeq1} exists for all complex time.  Let the origin be an equilibrium point of the vector field $V$. The origin is said to be {\em globally asymptotically stable } of the system \eqref{E:diffeq1} if for every  $\eps>0$ there exists $\delta>0$ such that $X_{t}(B(0,\delta)) \subset B(0,\eps)$ for all $t \geq 0$ and $\lim_{t \to \infty}X(t,z)=0$ for all $z \in \mathbb{C}^d$, where $X(t,z)$ denotes the solution of the system \eqref{E:diffeq1} with the initial condition $X(0)=z$. It is clear from the definition that if the origin is a globally asymptotic equilibrium point of the vector field, then it is unique. If a complete holomorphic vector field admits a globally asymptotically stable equilibrium point, then we say that the vector field is globally asymptotically stable.

\begin{definition}
Let $V \in \mathfrak{X}_{\mathcal{O}}(\mathbb{C}^d)$ be a complete vector field. A domain $D \subset \mathbb{C}^d$ is said to be \textit{spirallike} with respect to the vector field $V$ if $X_{t}(D) \subset D$ for all $t \geq 0$. We say $D$ is \textit{circularlike} with respect to the vector field $V$ if $X_{is}(D) \subseteq D$ for $s \in \left(-\frac{\pi}{2}, \frac{\pi}{2}\right)$.
\end{definition}

\smallskip

 \par Before stating our main result, let us first mention an example of a Kobayashi hyperbolic unbounded domain with an empty core that does not have any bounded realization. This example can be found in \cite[Example 3.9]{Poletsky2020}. For completeness, we include it here.

\begin{example}\label{Exam:Core_empty}

For a regular compact set $K$ in $\mathbb{C}$ with Hausdorff measure $H_{1}(K) = 0$ but $H_{\frac{1}{2}}(K) >0$, consider the domain $D = \mathbb{C} \setminus K$. The compact set $K$ is non-polar, and hence from \cite[Section 4.1]{Poletsky2020}, it follows that the domain $D$ is hyperconvex. Then, by \cite[Remark (ii), Page 604]{NikPflug2005}, the Green function of $D$ is continuous. Also, $D$ is a Greenian complex manifold. Therefore, the pluricomplex Green function $g_{D}$ satisfies the following (see \cite[Section 3]{Poletsky2020}): For any point $w_0 \in D$, there exist constants $C_1$ and $C_2$ such that the negative continuous plurisubharmonic function $g_{D}$ satisfies $\log |z - w_0| + C_1 < g_{D}(z,w_{0}) < \log |z - w_0| + C_2$ near $w_0$. Consequently, by \cite[Theorem 3.1]{Poletsky2020} (see also \cite[Theorem 2]{PoSh2019}), $\mathfrak{c}(D)$ becomes empty. Again, any bounded holomorphic function on $D$ extends holomorphically to $\mathbb{C}$ and, consequently, is constant. Thus, $D$ is not Carath$\acute{e}$odory hyperbolic; in particular, $D$ is not biholomorphic to any bounded domain in $\mathbb{C}$.

\end{example}

\smallskip

\par In this paper, we prove the following result, which gives a new class of domains for which the existence of a non-trivial core is the only obstruction for a bounded representation.

\begin{theorem}\label{T:Hyperbolic_Bounded}
		Let $D \subset \mathbb{C}^d$ be a domain containing the origin which is circularlike spirallike with respect to a complete, globally asymptotically stable vector field $V \in \mathfrak{X}_{\mathcal{O}}(\mathbb{C}^d)$ with $V(0)=0$. Then the following are equivalent:
		\begin{enumerate}[{\bf H(1)}]
			\item $D$ is bounded.
			\item $D$ is biholomorphic to a bounded domain.
			\item $D$ is Kobayashi hyperbolic.
			\item The core $\mathfrak{c}(D)$ of $D$  is empty.
     \item $D$ is Carath$\acute{e}$odory hyperbolic.
		\end{enumerate} 
	\end{theorem}

\smallskip

Let $A \in GL_{d}(\mathbb{C})$. Consider $A : \mathbb{C}^{d} \to \mathbb{C}^{d}$ as a holomorphic vector field. Then, as a corollary of \Cref{T:Hyperbolic_Bounded}, we obtain the following result:
	
	\begin{corollary}\label{Cor:Spiralike_Matrix}
Let $A$ be a Hurwitz matrix and $D$ be a domain in $\mathbb{C}^{d}$ containing the origin. Assume also that $e^{tA}z\in D$ and $e^{isA}z\in D$ whenever $z\in D$ for all $t\ge 0,$ for all $s\in (-\frac{\pi}{2},\frac{\pi}{2}).$ Then the following are equivalent:
	\begin{enumerate}[{\bf H(1)}]
		\item $D$ is bounded.
		\item $D$ is biholomorphic to a bounded domain.
		\item $D$ is Kobayashi hyperbolic.
		\item The core $\mathfrak{c}(D)$ of $D$  is empty.
   \item $D$ is Carath$\acute{e}$odory hyperbolic.
	\end{enumerate} 	
	\end{corollary}
	
\begin{remark}
If we take $A=-I_{d}$ in \Cref{Cor:Spiralike_Matrix}, we obtain a generalization of \cite[Theorem II]{Kod1982}. It is important to note that we do not require the full strength of the circular property of the domain in the case of \cite[Theorem II]{Kod1982}. \Cref{Cor:Spiralike_Matrix} suggests that it is sufficient to have $(e^{i\theta}z_{1}, \cdots, e^{i \theta}z_{d})\in D$ whenever $z\in D$ for $\theta \in (-\frac{\pi}{2}, \frac{\pi}{2})$.
\end{remark}

 \smallskip

In \cite{GalHarHer2017}, the authors considered a new class of plurisubharmonic functions to define a new core of a domain. They considered a slightly larger class of functions defined on a given domain $D \subset \mathbb{C}^d$:
\begin{align*}
    PSH^{'}(D):=\{\phi\in PSH(D): \phi\not\equiv -\infty, v(\phi,.)\equiv 0\},
\end{align*}
where $\text{PSH}(\Omega)$ denotes the family of plurisubharmonic functions in $D$ and $v(\phi,z_{0})$ denotes the Lelong number of $\Phi$ at $z_0$, i.e.
\begin{align*}
   v(\phi,.):= \liminf_{z \to z_0} \frac{\Phi(z)}{\log |z - z_0|}. 
\end{align*}
 
\noindent For the above class of plurisubharmonic functions, they provided the following definition of the core of a domain in $\mathbb{C}^d$:

\begin{definition}\label{Def:Core_Continuous}
Let $D \subset \mathbb{C}^d$ be a domain. Then, the set 
\begin{align*}
\mathfrak{c}^{'}(D) := \{\zeta \in D :\textit{every } \phi\in PSH^{'}(D) \textit{ that is bounded from above } \\ \textit{ fails to be smooth and strictly plurisubharmonic at } \zeta\}
\end{align*}
is called the core of $D$.\end{definition}

\smallskip

It is important to note that although the concept of the core of a domain depends on the smoothness of the class of plurisubharmonic functions, Richberg’s classical theorem (see Theorem I.5.21 in \cite{Richberg1968} or \cite{Demailly}) implies that the condition of having an empty core remains independent of the smoothness class as long as the functions are at least continuous. In view of \Cref{Def:Core_Smooth}, \Cref{R:Core_Empty_Hyperbolic} and \Cref{R:Core_Empty_Psh} establish a direct connection between Kobayashi hyperbolicity and the core of a domain. 
However, considering \Cref{Def:Core_Continuous} of the core of a domain makes the situation more intricate. For instance, there exists a domain $D$ with an empty $\mathfrak{c}^{'}(D)$ but non-empty $\mathfrak{c}(D)$, as seen in \cite{ShcherZhang2021}. Additionally, there exist domains with an empty $\mathfrak{c}^{'}(D)$ but are not Carath$\acute{e}$odory hyperbolic; however, they are Kobayashi hyperbolic, as shown in \cite{ShcherZhang2021}. With the above definition of the core of the domain, we conjecture that the existence of the nontrivial core is the only obstruction for bounded representations of certain classes of domains in $\mathbb{C}^d$.

	\begin{problem}\label{Con:Hyperbolic_Bounded}
Let $D$ be a circularlike and strictly spirallike domain with respect to the globally asymptotically stable vector field $V$ in $\mathbb{C}^d$. Then the following are equivalent:
		\begin{enumerate}[{\bf $H^{'}$(1)}]
			\item $D$ is bounded.
			\item $D$ is biholomorphic to a bounded domain.
			\item $D$ is Kobayashi hyperbolic.
            \item $D$ is Carath$\acute{e}$odory hyperbolic.
            \item The core $\mathfrak{c}^{'}(D)$ of $D$  is empty.
			%\item $D$ is Brody hyperbolic.
		\end{enumerate} 
	\end{problem}

\noindent \Cref{T:Hyperbolic_Bounded} tells us that $H'(j)$ is true if and only if $H'(j+1)$ is true for $j=1,2,3$. However, we do not know if $H'(5)$ implies any of the above.
	
	\section{Technical Results}\label{S:Technical}

In this section, we collect some basic and known technical results necessary for proving the theorems in this paper.	
	\begin{result}(\cite[Lemma 1]{Shcherbina2021})\label{R:Core_Empty_Psh}
		For a complex manifold M the property that the core c(M) of M is empty
		holds true if and only if there is a bounded smooth strictly plurisubharmonic function
		on M.
	\end{result}

	\begin{result}(\cite[Theorem 2]{Shcherbina2021})\label{R:Core_Empty_Hyperbolic}
		Let M be a complex manifold which has a bounded continuous strictly plurisubharmonic function. Then M is Kobayashi hyperbolic.
	\end{result}
	\begin{result}\cite[Proposition 3.1]{chatterjee2024approximations}\label{R:CG}
	    Let $V \in \mathfrak{X}_{\mathcal{O}}{(\mathbb{C}^{n})}$ be a complete globally asymptotically stable vector field such that $V(0)=0$ and $\Omega$ be a spirallike domain with respect to $V$ containing the origin. Then, for every compact subset $K \subset \mathbb{C}^{n}$, there exists a real number $M_{K}>0$ such that $X(t,z) \in \Omega$ for all $t >M_{K}$ and for all $z \in K$.
	\end{result}
	\begin{lemma}\label{L:Estmt_Ball_Sphere}
		Let $D$ be a Kobayashi hyperbolic domain in $\mathbb{C}^d$ containing the origin. For every $\eps>0$ with $S_{\eps}:=\{z \in D: \|z\|=\eps\}\subset D,$ there exists $\delta_{\eps}>0$ such that
		\begin{align*}
			K_{D}(0,x)\le K_{D}(0,y)~~,\forall {x}\in B_{\delta_\eps}:=\{z\in \mathbb{C}^d:\|z\|\le \delta_{\eps}\}, \forall y\in S_{\eps}.
		\end{align*}
	\end{lemma}
	\begin{proof}[Proof of \Cref{L:Estmt_Ball_Sphere}]
		Since $K_{D}$ is continuous, $\lim_{x\to 0} K_{D}(0,x) = 0$. We have  $\max_{B_{\delta_\eps}} K_{D}(0,x) = K_{D}(0,x_{\delta_\eps})$ for some $x_{\delta_\eps}$ with $|x_{\delta_{\eps}}| \le \delta_{\eps}$. Let us define: $R_{\eps} := \inf_{y\in S_{\eps}} K_{D}(0,y) > 0$. Choose $\delta_{\eps} > 0$ such that $K_{D}(0,x_{\delta_\eps}) < R_{\eps}$. Clearly, we get $\sup_{|x|\le \delta_{\eps}} K_{D}(0,x) = K_{D}(0,x_{\delta_\eps}) < R_{\eps} \le K_{D}(0,y)$ for all $y\in S_{\eps}$. Therefore, for all $x\in B_{\delta_\eps}$, we have $K_{D}(0,x) \le K_{D}(0,y)$ for all $y\in S_{\eps}$.
		
	\end{proof}

	\section{Proof of \Cref{T:Hyperbolic_Bounded}}
	\begin{proof}[Proof of \Cref{T:Hyperbolic_Bounded}]
		
	\noindent {\bf H(1)$\iff$ H(3)}:	
		Since every bounded domain is hyperbolic, we just need to show the converse part.
		
		\smallskip
		
		\noindent Let us assume that $D$ is unbounded. Then there exists a sequence of points $\{z_k\}_{k}\subset D$ such that $|z_{k}|>1$ for all $k\in \mathbb{N}.$ Let us consider the set $\mathbb{H}:=\{z\in \mathbb{C}: \rl (z)>0\},$ and  for each $k\in \mathbb{N}$ we define a function $f_{k}:\mathbb{H}\to D$ by 
		\begin{align*}
			f_{k}(\lambda)=X(\lambda,z_{k}).
		\end{align*}
		Clearly, if $f$ is well-defined, then it is a holomorphic function on $\mathbb{H}.$ It is easy to see that the map \( f_k \) is well-defined. In fact, for $\lambda \in \mathbb{H}$, we write $\lambda = t + is$, where $t > 0$ and $s \in (-\frac{\pi}{2}, \frac{\pi}{2})$. Therefore,
		\begin{align*}
			X(\lambda, z_k) = X(t + is, z_k) = X_t \circ X_{is}(z_k).
		\end{align*}
		Since $D$ is circularlike, $X_{is}(z_k) \in D$, and since $D$ is spirallike, $ X_t \circ X_{is}(z_k) \in D$. Therefore, $f_k$ is well-defined.

		\smallskip
		
		\noindent For sufficiently small $\eps \in (0, 1)$, let us consider the sphere
		\begin{align*}
			S_{\eps}:=\{z\in \mathbb{C}^d:\|z\|=\eps\}\subset D.
		\end{align*}
		Note that $|X(0,z_{k})| = |z_{k}| > \eps$. Again, since the vector field is globally asymptotically stable, there exists $T_{k} > 0$ such that $|X(T_{k},z_{k})| < \eps.$ Hence there exists $0<t_{k}<T_k$ such that $|X(t_k,z_k)| = \eps$.

		\smallskip
		
		\noindent Now \textbf{claim} that $t_{k} \to \infty $ as $ k \to \infty$. If not, then there exists $M > 0$ such that $t_{k} < M $ for all $k \in \mathbb{N}$. This implies, there exists $t_{0} > 0,$ and a subsequence $\{k_q\}$ of $\{k\}$ such that $t_{k_{q}} \to t_{0}$ as $q \to \infty$. Now $|z_{k_q}| = |X_{-t_{k_q}} \circ X_{t_{k_q}}(z_{k_q})|$. Since $X_{t_{k_q}}(z_{k_q}) \in S_{\eps}$, therefore, $X_{t_{k_{q_r}}}(z_{k_{q_{r}}}) \to b_{0}$  as $r\to\infty$ for some $b_{0} \in S_{\eps},$ and for some subsequence $\{k_{q_r}\}$ of $\{k_{q}\}.$ Since $ t_{k_q} \to t_{0}$, then $t_{k_{q_r}} \to t_{0}$.
		
		\smallskip
		
		\noindent We now compute:
		
		\begin{align*}
			\lim_{r\to \infty}\left|z_{k_{q_r}}\right|&=\lim_{r\to \infty}\left|X_{-t_{k_{q_r}}}\left(X_{t_{k_{q_r}}}\left(z_{k_{q_r}}\right)\right)\right|\\
			&=|X_{-t_{0}}(b_0)|.   
		\end{align*}
		But $\left|z_{k_{q_r}}\right|\to \infty$ as $r\to \infty.$ Hence this is a contradiction. Therefore $t_{k}\to \infty$ as $k\to \infty,$ which proves our claim.
		
		\smallskip
		
		\noindent Note that we have constructed  a sequence $t_{k}\in (0,\infty),z_{k}\in D$ such that $\left |X_{t_k}(z_k)\right|=\eps$ for all $k$.
		
		\begin{lemma}\label{L:Estimate_Rightplane}
	$\lim_{r\to\infty}K_{\mathbb{H}}(t_{k_{q_r}},t_{k_{q_r}}+\alpha)=0$ for any $\alpha\in \mathbb{R}^{+}.$
		\end{lemma}
		
		We accept this lemma as true at this point and proceed with the proof of the main theorem. Since $D$ is Kobayashi hyperbolic, hence,
		\begin{align*}
			K_{D}(f_{k}(t_k),0)&\le  K_{D}(f_{k}(t_k),f_{k}(t_k+\alpha))+K_{D}(f_{k}(t_k+\alpha),0)\\
			&\le  K_{\mathbb{H}}(t_k,t_k+\alpha)+K_{D}(f_{k}(t_k+\alpha),0).
		\end{align*}
		From above, for the sequence $\{k_{q_r}\}$, we have 
		\begin{align}\label{E:Estimate1}
			K_{D}\left(X\left(t_{k_{q_r}},z_{k_{q_r}}\right),0\right)& \leq  K_{\mathbb{H}}(t_{k_{q_r}},t_{k_{q_r}}+\alpha)+K_{D}(X(t_{k_{q_r}}+\alpha,z_{k_{q_r}}),0).
		\end{align}
		
		\noindent Taking limit $r\to \infty$ in \eqref{E:Estimate1}, we get from \Cref{L:Estimate_Rightplane} that 
		
		\begin{align}\label{E:Estimate2}
			K_{D}(b_0,0) \le \lim_{r\to \infty}K_{D}(X(t_{k_{q_r}}+\alpha,z_{k_{q_r}}),0).
		\end{align}
		
	\noindent	From \Cref{R:CG}, we can  choose $\alpha>0$ such that $X(\alpha,z)\in \overline{B(0,\delta_{\eps})}$ for all $z\in S_{\eps}$ and $\delta_{\eps}>0$ is chosen as \Cref{L:Estmt_Ball_Sphere}. Again, we have $X(t_{k_{q_r}}+\alpha,z_{k_{q_r}})=X(\alpha,X(t_{k_{q_r}},z_{k_{q_r}}))$, and by construction $X(t_{k_{q_r}},z_{k_{q_r}})\in S_{\eps}$. By the choice of $\alpha>0$, we get that
		$X(t_{k_{q_r}}+\alpha,z_{k_{q_r}})\in \overline{B(0,\delta_{\eps})}$. Therefore,
		\begin{align}\label{E:Estimate3}
			w:=\lim_{r \to \infty} X(t_{k_{q_r}}+\alpha,z_{k_{q_r}})\in \overline{B(0,\delta_{\eps})}.
		\end{align}
		\noindent Since $\delta_{\epsilon}>0$ in the above equation as per \Cref{L:Estmt_Ball_Sphere}, then $K_{D}(w,0)\le K_{D}(y,0)$ for all $y\in S_{\epsilon}$. Since $b_{0} \in S_{\epsilon}$, this contradicts \eqref{E:Estimate2}.
		
\medskip

	It remains to prove \Cref{L:Estimate_Rightplane}. Let us consider the biholomorphism $\phi: \mathbb{H} \to \mathbb{D}$ defined by $\phi(z) = \frac{z - 1}{z + 1}$, and let us consider the automorphism of the unit disc $\phi_{a}(z) = \frac{z - a}{1 - \overline{a}z}$ for $a \in \mathbb{D}$. Therefore,
		\begin{align*}
	K_{\mathbb{H}}(t_k,t_{k}+\alpha)
			&= K_{\mathbb{D}}(\phi(t_k),\phi(t_k+\alpha)) \\
			&= K_{\mathbb{D}}\left(\frac{t_k-1}{t_k+1},\frac{t_k+\alpha-1}{t_k+\alpha+1}\right) \\
			&= K_{\mathbb{D}}(\phi_{\alpha_{k}}(\alpha_{k}), \phi_{\alpha_{k}}(\beta_{k})), \\
   \intertext{where $\alpha_{k}=\frac{t_{k}-1}{t_{k}+1}$ and $\beta_{k}=\frac{t_{k}+\alpha-1}{t_{k}+\alpha+1}$. Hence, we get }
		K_{\mathbb{H}}(t_k,t_{k}+\alpha)	&= K_{\mathbb{D}}\left(0,\frac{\alpha}{2t_{k}+\alpha}\right).
		\end{align*}
\noindent
		Taking $k \to \infty$, we get the result.

\smallskip

\noindent {\bf H(3)$\iff$ H(4)}:
If $D$ is Kobayashi hyperbolic, then it is bounded, and thus $\mathfrak{c}(D)$ is empty. Conversely, if $\mathfrak{c}(D)$ is empty, then according to \Cref{R:Core_Empty_Hyperbolic}, there exists a bounded positive strictly plurisubharmonic function on $D$. Therefore, by \Cref{R:Core_Empty_Psh}, $D$ becomes Kobayashi hyperbolic.

\smallskip

\noindent {\bf H(2)$\implies$ H(4)}: Suppose $D$ is biholomorphic to a bounded domain $G$, and let $F:D\to G$ be a biholomorphism. Define $\phi(\xi):=\|F(\xi)\|^2$. Clearly, $\phi$ is a positive bounded above smooth strictly plurisubharmonic function on $D$, and hence $\mathfrak{c}(D)$ becomes empty. This proves the required implication.

\smallskip

\noindent Clearly, ${\bf H(1)\implies H(2)},$ and ${\bf H(5)\implies H(3)\implies H(1)\implies H(5)},$  hence the result is  proved.
   \end{proof}
	
\noindent {\bf Acknowledgements.}  The work of the first-named author is supported by a CSIR fellowship (File No-09/921(0283)/2019-EMR-I). The work of the second-named author is supported by an IoE-IISc Postdoctoral Fellowship (R(HR)(IOE-PDF)(MA)(GMM)-114).

	\bibliographystyle{plain}
	\bibliography{biblio.bib}

\end{document}